\documentclass[graybox]{svmult}

\usepackage{mathptmx}       
\usepackage{helvet}         
\usepackage{courier}        
\usepackage{type1cm}        
\usepackage{makeidx}         
\usepackage{graphicx}        
\usepackage{multicol}        
\usepackage[bottom]{footmisc}

\usepackage{amsmath}
\usepackage{amssymb}
\usepackage{amscd}
\usepackage{indentfirst}

\def\1{{\bf 1}}

\makeindex             

\begin{document}

\title*{On $e^*\theta$-regular and $e^*\theta$-normal Spaces}
\titlerunning{On $e^*\theta$-regular and $e^*\theta$-normal Spaces\dots} 
\author{Burcu Sünbül Ayhan} 
\institute{Burcu Sünbül Ayhan,\at Rectorate, Turkish-German University, Beykoz 34820, İstanbul, Turkey. \\
\email{brcyhn@gmail.com \& burcu.ayhan@tau.edu.tr}}

\maketitle
\abstract{The purpose of this study is to introduce a new class of regular spaces called $e^*\theta$-regular spaces which is a generalization of the class of $\beta \theta$-regular spaces \cite{caldas-jafari}. Also, we investigate some basic properties and several characterizations of $e^*\theta$-regular and $e^*\theta$-normal \cite{ayhan-ozkoc1} spaces. Moreover, some functions such as $e^{\ast}\theta$-closed function, generalized $e^{\ast}\theta$-closed function, generalized $e^{\ast}\theta$-closed function have been defined and studied. Furthermore, we obtain some preservation theorems.}

\section{Introduction}
\label{sec:1}
     
Functional analysis is the starting point for many different fields of mathematics. One of them abstracted the concept of open ball to open set and thus a topology. Hence, many of the results were generalized to a much larger setting.

Regular and normal space concepts, which has an important role in the classification of topological spaces, was introduced and studied respectively by Vietoris \cite{vietoris} and Tietze \cite{tietze}. $e^*\theta$-regular spaces which is a generalization of regular spaces and $e^*\theta$-normal spaces which has defined by Ayhan and Özkoç \cite{ayhan-ozkoc1} as a generalization of regular spaces has defined via $e^*$-$\theta$-open sets \cite{farhan-yang}. In recent years, many authors have studied several forms of regularity; such as $\beta\theta$-regularity \cite{caldas-jafari}, semiregularity \cite{maheshwari-prasad}, preregularity \cite{deeb-etc}, $b$-regularity \cite{park}, $e$-regularity \cite{ozkoc-aslim}, $\beta$-regularity \cite{noiri} and normality; such as $\beta\theta$-normality \cite{caldas-jafari}, seminormality \cite{maheshwari-prasad2}, prenormality \cite{paul-bhat}, $e$-normality \cite{ekici4}, $e^*$-normality \cite{ekici4}, $e^*\theta$-normality \cite{ayhan-ozkoc1}.

The main goal of this book chapter is to study the notions of $e^*\theta$-regular and $e^*\theta$-normal spaces and obtain some of their characterizations.

 \section{$e^*\theta$-regular and $e^*\theta$-normal Spaces}
 \label{sec:2}

Regular and normal spaces are two of the most important topological properties in topological spaces. In this chapter, the concepts of $e^*\theta$-regular and $e^*\theta$-normal \cite{ayhan-ozkoc1} space will be discussed. 
 
 \subsection{Preliminaries} 
 \label{subsec:2.1}

 Throughout the manuscript, $X$ and $Y$ represent topological spaces. For a subset $A$ of a space $X$, $cl(A)$ and $int(A)$ denote the closure of $A$ and the interior of $A$, respectively. The family of all closed (resp. open) sets of $X$ is denoted by $C(X)$ $(\text{resp. }O(X))$. A subset $A$ is said to be regular open \cite{stone} (resp. regular closed \cite{stone}) if $A=$ $int(cl(A))$ $($resp. $A=cl(int(A)))$. A point $x\in X$ is said to be $\delta $-cluster point \cite{velicko} of $A$ if $int(cl(U))\cap A\neq \emptyset $ for each open neighbourhood $U$ of $x$. The set of all $\delta $-cluster points of $A$ is called the $\delta $-closure \cite{velicko} of $A$ and is denoted by $cl_{\delta }(A)$. If $A=cl_{\delta }(A)$, then $A$ is called $\delta $-closed \cite{velicko} and the complement of a $\delta $-closed set is called $\delta $-open \cite{velicko}. The set $\{x|(\exists U \in O(X,x))(int(cl(U)) \subseteq A)\}$ is called the $\delta$-interior of $A$ and is denoted by $int_{\delta}(A)$.

A subset $A$ is called semiopen \cite{levine} (resp. preopen \cite{mashhour-monsef-deeb}, $b$-open \cite{andrijevic}, $\beta$-open \cite{monsef}, $e$-open \cite{ekici2}, $e^{\ast}$-open \cite{ekici3}) if $A$ $\subseteq cl(int(A))$ $($resp. $A$ $\subseteq int(cl(A))$, $A$ $\subseteq cl(int(A))\cup int(cl(A))$, $A\subseteq cl(int(cl(A)))$, $A$ $\subseteq cl(int_{\delta }(A))\cup int(cl_{\delta }(A))$, $A$ $\subseteq cl(int(cl_{\delta }(A))))$. The complement of a semiopen (resp. preopen, $b$-open, $\beta$-open, $e$-open, $e^{\ast}$-open) set is called semiclosed \cite{levine} (resp. preclosed \cite{mashhour-monsef-deeb}, $b$-closed \cite{andrijevic}, $\beta$-closed \cite{monsef}, $e$-closed \cite{ekici2}, $e^{\ast}$-closed \cite{ekici3}). The intersection of all semiclosed (resp. preclosed, $b$-closed, $\beta$-closed, $e$-closed, $e^{\ast}$-closed) sets of $X$ containing $A$ is called the semi-closure \cite{levine} (resp. pre-closure \cite{mashhour-monsef-deeb}, $b$-closure \cite{andrijevic}, $\beta$-closure \cite{monsef}, $e$-closure \cite{ekici2}, $e^{\ast}$-closure \cite{ekici3}) of $A$ and is denoted by $scl(A)$ $($resp. $pcl(A)$, $bcl(A)$, $\beta cl(A)$, $e$-$cl(A)$, $e^{\ast}$-$cl(A))$. The union of all semiopen (resp. preopen, $b$-open, $\beta$-open, $e$-open, $e^{\ast}$-open) sets of $X$ contained in $A$ is called the semi-interior \cite{levine} (resp. pre-interior \cite{mashhour-monsef-deeb}, $b$-interior \cite{andrijevic}, $\beta$-interior \cite{monsef}, $e$-interior \cite{ekici2}, $e^{\ast}$-interior \cite{ekici3}) of $A$ and is denoted by $sint(A)$ $($resp. $pint(A)$, $bint(A)$, $\beta int(A)$, $e$-$int(A)$, $e^{\ast}$-$int(A))$.

A point $x$ of $X$ is called an $e^{\ast }$-$\theta $-cluster ($\beta$-$\theta $-cluster) point of $A$ if $e^{\ast }$-$cl(U)\cap A\neq \emptyset$ $(\beta cl(U)\cap A\neq \emptyset)$ for every $e^{\ast}$-open ($\beta$-open) set $U$ containing $x$. The set of all $e^{\ast }$-$\theta $-cluster ($\beta$-$\theta $-cluster) points of $A$ is called the $e^{\ast }$-$\theta $-closure \cite{farhan-yang} ($\beta$-$\theta $-closure \cite{noiri}) of $A$ and is denoted by $e^{\ast }$-$cl_{\theta }(A)$ $(\beta cl_{\theta }(A))$. A subset $A$ is said to be $e^{\ast }$-$\theta $-closed $(\beta$-$\theta$-closed$)$ if $A=e^{\ast }$-$cl_{\theta }(A)$ $(A=\beta cl_{\theta }(A))$. The complement of an $e^{\ast }$-$\theta $-closed ($\beta$-$\theta $-closed) set is called an $e^{\ast }$-$\theta $-open \cite{farhan-yang} ($\beta$-$\theta $-open \cite{noiri}). A point $x$ of $X$ said to be an $e^{\ast}$-$\theta$-interior \cite{farhan-yang} ($\beta$-$\theta$-interior \cite{noiri}) point of a subset $A$, denoted by $e^{\ast}$-$int_{\theta }(A)$ $(\beta int_{\theta }(A))$, if there exists an $e^{\ast}$-open ($\beta$-open) set $U$ of $X$ containing $x$ such that $e^{\ast}$-$cl(U)\subseteq A$ $(\beta cl(U)\subseteq A)$. 
A subset $A$ is said to be $e^{\ast}$-$\theta$-regular \cite{farhan-yang} (resp. 
$\beta$-$\theta$-regular \cite{noiri}, $e^{\ast }$-regular \cite{farhan-yang}) set if it is $e^{\ast }$-$\theta$-open (resp. $\beta$-$\theta$-open, $e^{\ast }$-open) and $e^{\ast }$-$\theta$-closed (resp. $\beta$-$\theta$-closed, $e^{\ast }$-closed).

Also it is noted in \cite{farhan-yang} that 
\begin{center}
	$e^{\ast }\mbox{-regular}\Rightarrow e^{\ast }\mbox{-}\theta \mbox{-open}\Rightarrow e^{\ast }\mbox{-open.}$
\end{center}
The family of all $e^{\ast}$-$\theta$-open $($resp. $e^{\ast}$-$\theta$-closed, $e^{\ast}$-$\theta$-regular, $\beta$-$\theta$-open, $\beta$-$\theta$-closed, $e^{\ast}$-open, $e^{\ast}$-closed, $e^{\ast}$-regular, regular open, regular closed, $\delta$-open, $\delta$-closed, semi \text{\!}open, semiclosed, preopen, preclosed, $b$-open, $b$-closed, $\beta$-open, $\beta$-closed, $e$-open, $e$-closed$)$ subsets of $X$ is denoted by $e^{\ast} \theta O(X)$  $($resp. $e^{\ast} \theta C(X)$, $e^{\ast} \theta R(X)$, $\beta \theta O(X)$, $\beta \theta C(X)$, $e^{\ast} O(X)$, $e^{\ast} C(X)$, $e^{\ast} R(X)$, $RO(X)$, $RC(X)$, $\delta O(X)$, $\delta C(X)$, $SO(X)$, $SC(X)$, $PO(X)$, $PC(X)$, $BO(X)$, $BC(X)$, $\beta O(X)$, $\beta C(X)$, $eO(X)$, $eC(X))$. The family of all open $($resp. closed, $e^{\ast}$-$\theta$-open, $e^{\ast}$-$\theta$-closed, $e^{\ast}$-$\theta$-regular, $\beta$-$\theta$-open, $\beta$-$\theta$-closed, $e^{\ast}$-open, $e^{\ast}$-closed, $e^{\ast}$-regular, regular open, regular closed, $\delta$-open, $\delta$-closed, semiopen, semiclosed, preopen, preclosed, $b$-open, $b$-closed, $\beta$-open, $\beta$-closed, $e$-open, $e$-closed$)$ sets of $X$ containing a point $x$ of $X$ is denoted by $O(X,x)$ $($resp. $C(X,x)$, $e^{\ast}\theta O(X,x)$, $e^{\ast}\theta C(X,x)$, $e^{\ast}\theta R(X,x)$, $\beta \theta O(X,x)$, $\beta \theta C(X,x)$, $e^{\ast}O(X,x)$, $e^{\ast}C(X,x)$, $e^{\ast}R(X,x)$, $RO(X,x)$, $RC(X,x)$, $\delta O(X,x)$, $\delta C(X,x)$, $SO(X,x)$, $SC(X,x)$, $PO(X,x)$, $PC(X,x)$, $BO(X,x)$, $BC(X,x)$, $\beta O(X,x)$, $\beta C(X,x)$, $eO(X,x)$, $eC(X,x))$.

We shall use the well-known accepted language almost in the whole of the proofs of the theorems in this book chapter. The following basic properties of $e^{\ast }$-$\theta$-closure is useful in the sequel: 

\begin{lemma} \label{a}
{ \rm \cite{farhan-yang,jumaili-etc} 
Let $A$ be a subset of a space $X$, then the following properties hold:\newline
(1) $A \subseteq e^{\ast }$-$cl(A)$ $\subseteq$ $e^{\ast }$-$cl_{\theta}(A);$\\
(2) If $A \in e^{\ast }O(X),$ then $e^{\ast }$-$cl_{\theta}(A)=e^{\ast}$-$cl(A);$\\
(3) If $A \subseteq B$, then $e^{\ast }$-$cl_{\theta}(A) \subseteq e^{\ast }$-$cl_{\theta}(B);$\\
(4) $e^{\ast }$-$cl_{\theta}(A) \in e^{\ast }\theta C(X)$ and $e^{\ast }$-$cl_{\theta}(e^{\ast }$-$cl_{\theta}(A))=e^{\ast }$-$cl_{\theta}(A);$\\
(5) If $A \in e^{\ast }\theta C(X),$ then $e^{\ast }$-$cl_{\theta}(A)=A;$\\
(6) If $A \in  e^{\ast }\theta C(X),$ then $ \setminus A \in e^{\ast }\theta O(X);$\\ 
(7) If $A_{\alpha} \in e^{\ast}\theta C(X)$ for each $\alpha \in \Lambda$, then $\bigcap \{A_{\alpha}|\alpha \in \Lambda \} \in e^{\ast }\theta C(X);$\\
(8) If $A_{\alpha} \in e^{\ast }\theta O(X)$ for each $\alpha \in \Lambda$, then $\bigcup \{A_{\alpha}|\alpha \in \Lambda \} \in e^{\ast}\theta O(X);$\\
(9) $e^{\ast }$-$cl_{\theta}(A)$ = $\bigcap \{U|(A \subseteq U)(U \in e^{\ast }\theta C(X)) \};$\\
(10) $e^{\ast }$-$cl_{\theta}(\setminus A)$ $=\setminus e^{\ast }$-$int_{\theta}(A)$ and $e^{\ast }$-$int_{\theta}(\setminus A)$ $=\setminus e^{\ast }$-$cl_{\theta}(A).$}
\end{lemma}


\begin{lemma} \label{bb} {\rm \cite{munkres}  For a topological space $X$, the following statements are equivalent:\\
(1) $X$ is regular;\\
(2) $(\forall x \in X)[U \in O(X,x) \Rightarrow (\exists V \in O(X,x))(cl(V) \subseteq U)]$.}
\end{lemma}

\begin{definition} \label{q}
A topological space is said to be: \\
a) $p$-regular \cite{deeb-etc} if for each closed set $F$ of $X$ and each point $x \notin F$, there exist disjoint preopen sets $U$ and $V$ such that $F \subseteq U$ and $x \in V$.\\
b) $s$-regular \cite{maheshwari-prasad} if for each closed set $F$ of $X$ and each point $x \notin F$, there exist disjoint semiopen sets $U$ and $V$ such that $F \subseteq U$ and $x \in V$.\\
c) $b$-regular \cite{park} if for each closed set $F$ of $X$ and each point $x \notin F$, there exist disjoint $b$-open sets $U$ and $V$ such that $F \subseteq U$ and $x \in V$.\\
d) $\beta$-regular \cite{noiri} if for each closed set $F$ of $X$ and each point $x \notin F$, there exist disjoint $\beta$-open sets $U$ and $V$ such that $F \subseteq U$ and $x \in V$.\\
e) $\beta \theta$-regular \cite{caldas-jafari} if for each closed set $F$ of $X$ and each point $x \notin F$, there exist disjoint $\beta$-$\theta$-open sets $U$ and $V$ such that $F \subseteq U$ and $x \in V$.\\
f) $e$-regular \cite{ozkoc-aslim} if for each closed set $F$ of $X$ and each point $x \notin F$, there exist disjoint $e$-open sets $U$ and $V$ such that $F \subseteq U$ and $x \in V$.\\
g) $e^*$-regular if for each closed set $F$ of $X$ and each point $x \notin F$, there exist disjoint $e^*$-open sets $U$ and $V$ such that $F \subseteq U$ and $x \in V$.
\end{definition}

 \subsection{$e^*\theta$-regular and $(e^{\ast},\theta)$-regular Spaces} 
 \label{subsec:2.2}

In this section, two new classes of regular spaces called $e^*\theta$-regular and $(e^{\ast},\theta)$-regular spaces are introduced and studied through the concept of $e^*$-$\theta$-open set defined by Farhan and Yang \cite{farhan-yang}. In addition, not only some of their characterizations, but also their relations with other existing generalized regular space types are investigated.

\begin{definition} \label{c}
A topological space is said to be $e^*\theta$-regular if for each closed set $F$ of $X$ and each point $x \notin F$, there exist disjoint $e^*$-$\theta$-open sets $U$ and $V$ such that $F \subseteq U$ and $x \in V$.
\end{definition}

\begin{remark} \label{eb}
From Definitions \ref{q} and \ref{c}, we have following the diagram. However, none of the below implications is reversible as shown in the relevant articles and the following example.
\[
\begin{array}
[c]{ccccccc}%
p\text{-regular} & \longrightarrow & e\text{-regular} & \longrightarrow &
e^{\ast}\text{-regular} & \longleftarrow & e^{\ast}\theta\text{-regular}\\
& \searrow &  &  & \uparrow &  & \uparrow\\
s\text{-regular} & \longrightarrow & b\text{-regular} & \longrightarrow &
\beta\text{-regular} & \longleftarrow & \beta\theta\text{-regular} 
\end{array}
\]
\end{remark}

\begin{example}
Let $X = \{1, 2, 3, 4\}$ and $\tau= \{\emptyset,X,\{1\},\{2\}, \{1,2\},\{1,3,4\}\}$. It is not difficult to see $e^*\theta O(X)=2^{X}$ and $\beta \theta O(X)= \{\emptyset, X, \{2\},\{1,3,4\}\}$. Then, $X$ is $e^*\theta$-regular but it is not $\beta \theta$-regular.
\end{example}
\noindent
\textbf{Question:}
Is there any $e^*$-regular topological space which is not $e^*\theta$-regular space?


\begin{theorem} \label{1}
For a topological space $X$, the following statements are equivalent:\\
(1) $X$ is $e^*\theta$-regular;\\
(2) For each point $x$ in $X$ and each open set $U$ of $X$ containing $x$, there exists $V\in e^{\ast }\theta O(X,x)$ such that $e^{\ast}$-$cl_{\theta}(V)\subseteq U$;\\
(3) For each closed set $F$ of $X$,  $F=\bigcap \{e^{\ast }$-$cl_{\theta}(V)|(F \subseteq V)(V \in e^{\ast }\theta O(X)) \}$; \\
(4) For each subset $A$ of $X$ and each open set $U$ of $X$ such that $A \cap U \neq \emptyset$, there exists $V\in e^{\ast }\theta O(X)$ such that $A \cap V \neq \emptyset$
and $e^{\ast}$-$cl_{\theta}(V)\subseteq U$;\\
(5) For each nonempty subset $A$ of $X$ and each closed set $F$ of $X$ such that $A \cap F = \emptyset$, there exist $V,W\in e^{\ast }\theta O(X)$ such that $A \cap V \neq \emptyset$, $F \subseteq W$ and $V \cap W = \emptyset$.
\end{theorem}

\begin{proof}
$(1) \Rightarrow (2):$ Let $U\in O(X,x).$\\
    $\left. \begin{array}{rr}
	U\in O(X,x)\Rightarrow ( \setminus U \in C(X))(x \in U) \\
	\text{Hypothesis}
	\end{array} \right\} \Rightarrow$
	\\
	$\begin{array}{l}
	\Rightarrow (\exists V \in e^*\theta O(X,x))(\exists W \in e^*\theta O(X))( \setminus U \subseteq W)(V \cap W = \emptyset)
	\end{array}$
	\\
	$\begin{array}{l}
	\Rightarrow (\exists V \in e^*\theta O(X,x))(\exists W \in e^*\theta O(X))( \setminus U \subseteq W)(e^*\text{-}cl_\theta(V) \cap W = \emptyset)
	\end{array}$
\\
	$\begin{array}{l}
	\Rightarrow (\exists V \in e^*\theta O(X,x))(e^*\text{-}cl_\theta(V)\subseteq U).
	\end{array}$
\\

$(2) \Rightarrow (3):$ Let $F\in C(X)$ and $x \notin F.$\\
$\begin{array}{r}	
     F \in C(X) \Rightarrow F \subseteq \bigcap \{e^*\text{-}cl_\theta(V)|(F \subseteq V)(V \in e^*\theta O(X)) \} \ldots(1)
    \end{array}	$
\\
	$\left. 
	\begin{array}{r}
	(F\in C(X))(x \notin F)\Rightarrow \setminus F \in O(X,x)\\ 
	\text{Hypothesis}
	\end{array}%
	\right\} \Rightarrow 
	\\ 
	\left. 
	\begin{array}{r}
    \Rightarrow (\exists U\in e^{\ast }\theta O(X,x))(U \subseteq e^{\ast }\text{-}
	cl_{\theta }(U)\subseteq  \setminus F)  \\ 
	V:= \setminus e^*\text{-}cl_{\theta}(U)
	\end{array}
	\right\} \overset{\text{Lemma} \ \ref{a}}{\Rightarrow} $
	\\
    $\begin{array}{r}	
    \Rightarrow (F \subseteq V)(V\in e^{\ast }\theta O(X))(x \notin e^*\text{-}cl_{\theta}(V))
    \end{array}	$
    \\
    $\begin{array}{r}	
    \Rightarrow x \notin \bigcap \{e^*\text{-}cl_\theta(V)|(F \subseteq V)(V \in e^*\theta O(X)) \}
    \end{array}	$
    \\
    $\begin{array}{r}	
    \Rightarrow \bigcap \{e^*\text{-}cl_\theta(V)|(F \subseteq V)(V \in e^*\theta O(X)) \}\subseteq F \ldots(2) 
    \end{array}	$
    \\
    $\begin{array}{r}	
    (1),(2)\Rightarrow F = \bigcap \{e^*\text{-}cl_\theta(V)|(F \subseteq V)(V \in e^*\theta O(X)) \}.
    \end{array}	$ \\
    
$(3) \Rightarrow (4):$  Let $A \subseteq X$, $U\in O(X)$ and $A \cap U \neq \emptyset$.\\
    $
	\left.
	\begin{array}{r}
	U\in O(X)\Rightarrow  \setminus U \in C(X)  \\
	A \cap U \neq \emptyset \Rightarrow (\exists x \in X)(x \in A \cap U) \Rightarrow x \in U
	\end{array}
	\right\} \overset{\text{Hypothesis}}{\Rightarrow} $
	\\
    $
	\left.
	\begin{array}{r}
    \Rightarrow (\exists W \in e^*\theta O(X))( \setminus U \subseteq W)(x \notin e^*\text{-}cl_\theta(W))\\
	V:= \setminus e^*\text{-}cl_\theta(W)
	\end{array}
	\right\} \overset{\text{Lemma} \ \ref{a}}{\Rightarrow} $
	\\
	$\begin{array}{r}
	\Rightarrow ( V \in e^*\theta O(X))(x \in V \cap A \neq \emptyset)(e^*\text{-}cl_\theta(V)\subseteq e^*\text{-}cl_\theta( \setminus W) =  \setminus W \subseteq U).
	\end{array}$
    \\

$(4) \Rightarrow (5):$ Let $F\in C(X)$, $A \neq \emptyset$ and $A \cap F = \emptyset$.\\
    $
	\left.
	\begin{array}{r}
	F \in C(X)\Rightarrow  \setminus F \in O(X)  \\
	(A \neq \emptyset)(A \cap F= \emptyset) \Rightarrow  A \cap  ( \setminus F) \neq \emptyset
	\end{array}
	\right\} \overset{\text{Hypothesis}}{\Rightarrow} $
	\\
	 $
	\left.
	\begin{array}{r}
    \Rightarrow (\exists V \in e^*\theta O(X))(A \cap V \neq \emptyset)( e^*\text{-}cl_\theta(V) \subseteq \setminus F)\\
	W:= \setminus e^*\text{-}cl_\theta(V)
	\end{array}
	\right\} \Rightarrow $
	\\
	$\begin{array}{r}
	\Rightarrow ( V,W \in e^*\theta O(X))(A \cap V \neq \emptyset)(F \subseteq W)(V \cap W = \emptyset).
	\end{array}$\\
	
	$(5) \Rightarrow (1):$ It can be easily proved.
\end{proof}

\begin{definition} \label{d}
A topological space is said to be:\\
$a)$ $(e^{\ast},\theta)$-regular if for each $e^{\ast}$-$\theta$-regular set $F$ of $X$ and each point $x$ in $X \setminus F$, there exist disjoint open sets $U$ and $V$ such that $F \subseteq U$ and $x \in V$.\\
$b)$ Extremally $e^{\ast}\theta$-disconnected (briefly, $ED^{e^{\ast}\theta}$) if $e^*\mbox{-}cl_\theta(U)$ is $e^{\ast}$-$\theta$-open in $X$ for every $e^{\ast}$-$\theta$-open set $U$ of $X$.
\end{definition}

\begin{remark}
From  Definitions \ref{c} and \ref{d}, every  regular and $(e^{\ast},\theta)$-regular spaces are $e^*\theta$-regular space. The  converses  of  these  implications  are  not  true  in  general  as  shown  by the following example.
\end{remark}

\begin{example}
Let $X = \{a, b, c, d\}$ and $\tau= \{\emptyset,X,\{a\},\{b\}, \{a,b\},\{a,b,c\}\}$. It is not difficult to see $e^*\theta O(X)=e^* O(X)=2^{X} \setminus \{\{c\}, \{d\},\{c,d\}\}$.
Then $X$ is $e^*\theta$-regular but it is not regular and $(e^*,\theta)$-regular.
\end{example}

\begin{theorem} \label{2}
For a topological space $X$, the following statements are equivalent:\\
(1) $X$ is $(e^*,\theta)$-regular;\\
(2) For each $x$ in $X$ and any $e^{\ast}$-$\theta$-regular set $U$ of $X$ containing $x$, there exists $V\in O(X,x)$ such that $cl(V)\subseteq U$.
\end{theorem}

\begin{proof}
$(1) \Rightarrow (2):$ Let $x \in X$ and $U\in e^*\theta R(X,x).$\\
     $
	\left.
	\begin{array}{r}
	(x \in X)(U\in e^*\theta R(X,x))\Rightarrow ( \setminus U \in e^*\theta R(X))(x \in U)  \\
	X \mbox{ is } (e^{\ast},\theta)\mbox{-regular}
	\end{array}
	\right\} \Rightarrow$
	\\
	$\begin{array}{r}
	\Rightarrow (\exists W \in  O(X))(\exists V \in O(X,x))( \setminus U \subseteq W)(V \cap W = \emptyset)
	\end{array}$
	\\
	$\begin{array}{r}
	\Rightarrow (\exists W \in  O(X))(\exists V \in O(X,x))( \setminus U \subseteq W)(cl(V) \cap W = \emptyset)
	\end{array}$
\\
	$\begin{array}{r}
	\Rightarrow (\exists W \in  O(X))(\exists V \in O(X,x))(cl(V) \subseteq  \setminus W \subseteq U).
	\end{array}$
\\

$(2) \Rightarrow (1):$ Let  $F\in e^*\theta R(X)$ and $x \notin F$.\\
     $
	\left. \begin{array}{r}
	x \notin F \in e^*\theta R(X)\Rightarrow U:= \setminus F\in e^*\theta R(X,x) \\ \text{Hypothesis}
	\end{array} \right\} \Rightarrow    $
	\\
	$ \begin{array}{l}
	\Rightarrow (\exists V \in  O(X,x))(cl(V) \subseteq U)
	\end{array} $ 
	\\
	$\left.  \begin{array}{rr}
	\Rightarrow (\exists V \in  O(X,x))(F = \setminus U \subseteq  \setminus cl(V))\\ 
 W:=\setminus cl(V)
 \end{array} \right\}\Rightarrow$ 
	\\
	$\begin{array}{r}
	\Rightarrow (\exists V \in  O(X,x))(W \in O(X))(F \subseteq W)(V \cap W= \emptyset).
	\end{array}$
\end{proof}

\begin{theorem} \label{3}
If $X$ is $e^*\theta$-regular, $(e^*,\theta)$-regular and $ED^{e^*\theta}$, then it is regular.
\end{theorem}

\begin{proof}
Let $x \in O(X)$ and $x \in U.$\\
$
\left.\begin{array}{rr} 
U\in O(X,x) \\ X \mbox{ is } e^{\ast}\theta\mbox{-regular} \end{array}	\right\}\overset{\text{Theorem } \ref{1}} {\Rightarrow} \!\!\!\!\! \begin{array}{c} \\	\left. \begin{array}{r}  (\exists V \in e^*\theta O(X,x))(e^{\ast}\mbox{-}cl_\theta(V)\subseteq U)   \\ X \mbox{ is } ED^{e^{\ast}\theta} \end{array}	\right\} \Rightarrow \end{array}
$\\
	$
	\left.
	\begin{array}{r}
	\Rightarrow e^{\ast}\mbox{-}cl_\theta(V) \in e^*\theta R(X)\\
	X \mbox{ is } (e^{\ast},\theta)\mbox{-regular}
	\end{array}
	\right\}  \overset{\text{Theorem} \mbox{ } \ref{2}} {\Rightarrow}$ 
	\\
	$\begin{array}{r}
	\Rightarrow (\exists W \in O(X,x))(cl(W) \subseteq e^{\ast}\mbox{-}cl_\theta(V)\subseteq U)
	\end{array}$\\
Then, $X$ is regular from Lemma \ref{bb}.
\end{proof}

\begin{definition} \label{e}
A function $f:X \rightarrow Y$ is said to be: \\
$a)$ $e^{\ast}\theta$-closed (resp. pre $e^{\ast}\theta$-closed \cite{ayhan-ozkoc2}), if the image of each closed (resp. $e^{\ast}$-$\theta$-closed) set $F$ of $X$ is $e^{\ast}$-$\theta$-closed in $Y$.\\
$b)$ $e^{\ast}\theta$-open (resp. pre $e^{\ast}\theta$-open \cite{ayhan-ozkoc2}), if the image of each open (resp. $e^{\ast}$-$\theta$-open) set $U$ of $X$ is $e^{\ast}$-$\theta$-open in $Y$.
\end{definition}


\begin{lemma} \label{4}
A function $f:X \rightarrow Y$ is $e^*\theta$-closed (resp. pre $e^*\theta$-closed) if and only if for each subset $B$ of $Y$ and each open (resp. $e^*$-$\theta$-open) set $U$ containing $f^{-1}[B]$, there exists an $e^*$-$\theta$-open set $V$ of $Y$ containing $B$ such that $f^{-1}[V] \subseteq U$.
\end{lemma}

\begin{proof}
$(\Rightarrow):$ Let $B \subseteq Y$, $U \in O(X)$ and $f^{-1}[B] \subseteq U$.\\
	$
	\left.
	\begin{array}{r}
	(U \in O(X))(f^{-1}[B] \subseteq U) \Rightarrow (\setminus U \in C(X))( \setminus U \subseteq  \setminus f^{-1}[B]=f^{-1}[ \setminus B]) \\ 
	f \text{ is $e^*\theta$-closed} 
	\end{array}
	\right\} \Rightarrow 
	$
	\\
	$\left. \begin{array}{rr}
	\Rightarrow (f[ \setminus U] \in e^* \theta C(Y))(f[ \setminus U] \subseteq f[f^{-1}[ \setminus B]] \subseteq  \setminus B)\\
 V:= \setminus f[\setminus U]
	\end{array} \right\} \Rightarrow $ 
	\\
	$\begin{array}{r}
	\Rightarrow (V \in e^* \theta O(Y))( B\subseteq \setminus f[f^{-1}[\setminus B]] \subseteq V)(f^{-1}[V] \subseteq U).
	\end{array}$
	\\
	
	$(\Leftarrow):$ Let $F \in C(X)$.\\
		$\begin{array}{r}
		F \in C(X) \Rightarrow (f[F] \subseteq Y)(\setminus F \in O(X))(F \subseteq  f^{-1}[f[F]])
		\end{array}$
		\\
		$\left. \begin{array}{rr}
		\Rightarrow (\setminus f[F] \subseteq Y)(\setminus F \in O(X))( f^{-1}[ \setminus f[F]] \subseteq \setminus F) \\
		\text{Hypothesis}
		\end{array} \right\} \Rightarrow $
		\\
		$\begin{array}{r}
		\Rightarrow  (\exists V \in e^* \theta O(Y))( \setminus f[F] \subseteq V)( f^{-1}[V] \subseteq  \setminus F)
		\end{array}$
		\\
		$\begin{array}{r}
		\Rightarrow (\exists V \in e^* \theta O(Y))( \setminus V \subseteq f[F])(F \subseteq  \setminus f^{-1}[V] =  f^{-1}[ \setminus V])
		\end{array}$
		\\
		$\begin{array}{r}
		\Rightarrow (\exists V \in e^* \theta O(Y))( \setminus V \subseteq f[F])(f[F] \subseteq  f[f^{-1}[ \setminus V]] \subseteq  \setminus V)
		\end{array}$
		\\
		$\begin{array}{r}
		\Rightarrow (  \setminus V \in e^* \theta C(Y))(  \setminus V \subseteq f[F] \subseteq  \setminus V)
		\end{array}$
		\\
		$\begin{array}{r}
		\Rightarrow  f[F] \in e^* \theta C(Y).
		\end{array}$\\
		The other case can be proved similarly.
\end{proof}



\begin{definition} \label{f}
A subset $A$ of a space $X$ is called generalized $e^*\theta$-closed (briefly, $ge^*\theta$-closed) if $e^*$-$cl_{\theta}(A) \subseteq U$ whenever $A \subseteq U$ and $U$ is open in $X$. A subset $A$ of a space $X$ is said to be generalized $e^*\theta$-open (briefly, $ge^*\theta$-open) if $X \setminus A$ is $ge^*\theta$-closed. The family of all $ge^*\theta$-closed (resp. $ge^*\theta$-open) subsets of $X$ is denoted by $ge^{\ast}\theta C(X)$ $(\mbox{resp. } ge^{\ast}\theta O(X)).$
\end{definition}

\begin{lemma} \label{6}
A subset $A$ of a space $X$ is $ge^*\theta$-open if and only if $F \subseteq e^*$-$int_{\theta}(A)$ whenever $F \subseteq A$ and $F$ is closed in $X$.
\end{lemma}

\begin{proof}
$(\Rightarrow):$ Let $F \subseteq A$, $F \in C(X)$ and $A \in ge^{\ast}\theta O(X)$.\\
	$
	\left.
	\begin{array}{r}
	A \supseteq F \in C(X) \Rightarrow  \setminus A \subseteq  \setminus F\in O(X)\\ 
	A \in ge^{\ast}\theta O(X) \Rightarrow  \setminus A \in ge^{\ast}\theta C(X)
	\end{array}
	\right\} \Rightarrow 
	$
    \\
	$\begin{array}{r}
	\Rightarrow \setminus e^*\text{-}int_{\theta}(A)=e^*\text{-}cl_{\theta}( \setminus A) \subseteq  \setminus F
	\end{array}$
	\\
	$\begin{array}{r}
	\Rightarrow F \subseteq  e^*\text{-}int_{\theta}(A).
	\end{array}$
	\\
	$(\Leftarrow):$ Let $ \setminus F \in O(X)$ and $ \setminus A \subseteq  \setminus F$.\\
	$
	\left.
	\begin{array}{r}
	( \setminus F \in O(X))( \setminus A \subseteq  \setminus F) \Rightarrow  (F \in C(X))(F \subseteq A)\\ 
	\mbox{Hypothesis}
	\end{array}
	\right\} \Rightarrow 
	$
	\\
	$\begin{array}{r}
	\Rightarrow   F \subseteq  e^*\text{-}int_{\theta}(A)
	\end{array}$\\
	$\begin{array}{r}
	\Rightarrow  e^*\text{-}cl_{\theta}( \setminus A)=  \setminus e^*\text{-}int_{\theta}(A) \subseteq  \setminus F
	\end{array}$\\
	Then, $ \setminus A \in ge^{\ast}\theta C(X)$. Therefore, $A \in ge^{\ast}\theta O(X)$.
\end{proof}

\begin{definition} \label{g}
A function $f:X \rightarrow Y$ is said to be: \\
$a)$ generalized $e^{\ast}\theta$-closed (briefly, $ge^*\theta$-closed), if the image of each closed set $F$ of $X$ is $ge^{\ast}\theta$-closed in $Y$;\\
$b)$ generalized $e^{\ast}\theta$-open (briefly, $ge^*\theta$-open), if the image of each open set $U$ of $X$ is $ge^{\ast}\theta$-open in $Y$;\\
$c)$ pre generalized $e^{\ast}\theta$-closed (briefly, pre $ge^*\theta$-closed), if the image of each $ge^{\ast}\theta$-closed set $F$ of $X$ is $ge^{\ast}\theta$-closed in $Y$;\\
$d)$ pre generalized $e^{\ast}\theta$-open (briefly, pre $ge^*\theta$-open), if the image of each $ge^{\ast}\theta$-open set $U$ of $X$ is $ge^{\ast}\theta$-open in $Y$.
\end{definition}
    
\begin{lemma} \label{7}
For a function $f:X \rightarrow Y$, the following statements are equivalent:\\
(1) $f$ is $ge^*\theta$-closed (resp. pre $ge^*\theta$-closed);\\
(2) For each subset $B$ of $Y$ and each open (resp. $ge^*\theta$-open) set $U$ of $X$ containing $f^{-1}[B]$, there exists a $ge^*\theta$-open set $V$ of $Y$ containing $B$ such that $f^{-1}[V] \subseteq U$.
\end{lemma}

\begin{proof}
It is similar to the proof of Lemma \ref{4}.
\end{proof}

\begin{theorem} \label{8}
If $f : X \rightarrow Y$ is a continuous $e^*\theta$-open and $ge^*\theta$-closed surjection from a regular space $X$ onto a space $Y$, then $Y$ is $e^*\theta$-regular.
\end{theorem}

\begin{proof}
Let $x \in X$, $y=f(x)$ and $U \in O(Y,y)$.\\
$
\left.\begin{array}{rr} 
(y=f(x))(U \in O(Y,y)) \\ \text{$f$ is continuous} \end{array}	\right\}\Rightarrow \!\!\!\!\! \begin{array}{c} \\	\left. \begin{array}{r} f^{-1}[U]\in O(X,x)  \\ X \text{ is regular} \end{array}	\right\} \overset{\text{Lemma } \ref {bb}}{\Rightarrow} \end{array}
$\\
	$
	\left.
	\begin{array}{r}
	 \Rightarrow (\exists V \in O(X,x))(cl(V) \subseteq f^{-1}[U]) \\
	f \text{ is an $e^*\theta$-open surjection} 
	\end{array}
	\right\}\Rightarrow $
	\\
		$
	\left.
	\begin{array}{r}
	 \Rightarrow (f[V] \in e^*\theta O(Y,y))(f[cl(V)] \subseteq f[f^{-1}[U]]=U) \\
	f \text{ is $ge^*\theta$-closed} 
	\end{array}
	\right\} \Rightarrow  $
	\\
	$\begin{array}{r}
	\Rightarrow (f[cl(V)] \in ge^*\theta C(Y,y))(e^*\text{-}cl_\theta(f[V]) \subseteq  e^*\text{-}cl_\theta(f[cl(V)]) =  f[cl(V)] \subseteq U)
	\end{array}$
	\\
Then, $Y$ is  $e^*\theta$-regular from Theorem \ref{1}(2).
\end{proof}


 \subsection{$e^*\theta$-normal and $(e^{\ast},\theta)^*$-normal Spaces} 
 \label{subsec:2.3}

In this section, the fundamental properties of $e^*\theta$-normal spaces \cite{ayhan-ozkoc1} and $(e^{\ast},\theta)^*$-normal spaces are investigated. Additionally, several characterizations of both spaces have been obtained.

\begin{definition} \label{h}
A topological space is said to be $e^*\theta$-normal \cite{ayhan-ozkoc1} if for any pair of disjoint closed sets $F_1$ and $F_2$ of $X$, there exist disjoint $e^*$-$\theta$-open sets $U_1$ and $U_2$ such that $F_1 \subseteq U_1$ and $F_2 \subseteq U_2$.
\end{definition}

\begin{remark} \label{be}
From Definition \ref{h}, every normal space is an $e^*\theta$-normal space. The converse of this implication is not true in general as shown by the following example.
\end{remark}

\begin{example}
Let $X = \{a, b, c, d\}$ and $\tau=\{\emptyset,\{a\},\{b\},\{a,b\},\{a,c\},\{a,b,c\},\{a,b,d\}, \\ X\}$. It is not difficult to see that $e^*\theta O(X)=e^* O(X)=2^{X} \setminus \{\{d\}\}$. Then, $X$ is $e^*\theta$-normal but it is not normal.
\end{example}

\begin{theorem} \label{9}
For a topological space $X$, the following statements are equivalent:\\
(1) $X$ is $e^*\theta$-normal;\\
(2) For every pair of open sets $U$ and $V$ whose $U \cup V=X$, there exist $e^{\ast }\theta$-closed sets $A$ and $B$ such that $A\subseteq U$, $B\subseteq V$ and $A \cup B=X$;\\
(3) For each closed set $F$ and every open set $G$ containing $F$, there exists $U\in e^{\ast }\theta O(X)$ such that $F \subseteq U \subseteq e^{\ast}$-$cl_{\theta}(U)\subseteq G$.
\end{theorem}

\begin{proof}
$(1) \Rightarrow (2):$ Let $U,V \in O(X)$ and $U \cup V = X$.\\
	$
	\left.
	\begin{array}{r}
	(U,V \in O(X))(U \cup V = X) \Rightarrow  (\setminus U, \setminus V \in C(X))( ( \setminus U) \cap ( \setminus V) =\emptyset)  \\ 
	X \text{ is $e^*\theta$-normal}
	\end{array}
	\right\} \Rightarrow 
	$
	\\
	$
	\left.
	\begin{array}{r}
    \Rightarrow (\exists U_1 \in e^* \theta O(X))(\exists V_1 \in e^* \theta O(X))( \setminus U \subseteq U_1)( \setminus V \subseteq V_1)(U_1 \cap V_1 = \emptyset) \\ 
	(A:=\setminus U_1)(B:=\setminus V_1)
	\end{array}
	\right\} \Rightarrow 
	$
	\\
	$\begin{array}{r}
	\Rightarrow (A,B \in e^* \theta C(X))(A\subseteq U)(B\subseteq V)(A \cup B=X).
	\end{array}$
	\\

$(2) \Rightarrow (3):$ Let $F \in C(X)$, $G \in O(X)$ and $F \subseteq G$.\\
	$
	\left.
	\begin{array}{r}
	(F \in C(X)) (G \in O(X))(F \subseteq G) \Rightarrow 
	(\setminus F \in O(X))(G \in O(X)) (( \setminus F) \cup G = X)\\ 
	\text{Hypothesis}
	\end{array}
	\right\} \Rightarrow
	$
	\\
	$\begin{array}{r}
	\Rightarrow (\exists W_1,W_2 \in e^* \theta C(X))(W_1 \subseteq \setminus F)(W_2 \subseteq G )(W_1 \cup W_2 = X)
	\end{array}$
	\\	
	$
	\left.
	\begin{array}{r}
   \! \Rightarrow ( \setminus W_1, \setminus W_2 \in e^* \theta O(X))(F \subseteq \setminus W_1)( \setminus G \subseteq  \setminus W_2)(( \setminus W_1) \cap ( \setminus W_2) = \emptyset) \\ 
	(U:= \setminus W_1)(V:= \setminus W_2)
	\end{array}
	\! \right\} \! \Rightarrow 
	$
	\\
	$
	\begin{array}{r}
    \Rightarrow (U \in e^*\theta O(X))(\setminus V \in e^*\theta C(X))(F \subseteq U \subseteq  \setminus V \subseteq G)
	\end{array}
	$
	\\
	$\begin{array}{r}
	\Rightarrow (U \in e^*\theta O(X))(F \subseteq U \subseteq e^{\ast}\mbox{-}cl_{\theta}(U)\subseteq e^{\ast}\mbox{-}cl_{\theta}(\setminus V) = \setminus V \subseteq G).
	\end{array}$\\
	
$(3) \Rightarrow (1):$ Let $F_1,F_2 \in C(X)$ and $F_1 \cap F_2 = \emptyset$.\\
	$
	\left.
	\begin{array}{rr}
	(F_1,F_2 \in C(X)) (F_1 \cap F_2 = \emptyset ) \Rightarrow (\setminus F_1, \setminus F_2 \in O(X)) (F_1 \subseteq  \setminus F_2)\\
	\text{Hypothesis}
	\end{array}
	\right\} \Rightarrow
	$
	\\
	$
	\left.
	\begin{array}{r}
    \Rightarrow (\exists U \in e^* \theta O(X))(F_1 \subseteq U \subseteq e^{\ast}\mbox{-}cl_{\theta}(U)\subseteq  \setminus F_2) \\ 
	V:= \setminus e^{\ast}\mbox{-}cl_{\theta}(U) 
	\end{array}
	\right\} \Rightarrow 
	$
	\\
	$\begin{array}{r}
	\Rightarrow (U,V \in e^* \theta O(X))(F_1 \subseteq U)(F_2 \subseteq V)(U \cap V=\emptyset).
	\end{array}$
\end{proof}

\begin{theorem} \label{10}
For a topological space $X$, the following statements are equivalent:\\
(1) $X$ is $e^*\theta$-normal;\\
(2) For every pair of open sets $U$ and $V$ whose $U \cup V=X$, there exist $ge^{\ast }\theta$-closed sets $A$ and $B$ such that $A\subseteq U$, $B\subseteq V$ and $A \cup B=X$;\\
(3) For each closed set $F$ and every open set $G$ containing $F$, there exists $U\in ge^{\ast }\theta O(X)$ such that $F \subseteq U \subseteq e^{\ast}$-$cl_{\theta}(U)\subseteq G$;\\
(4) For each pair of disjoint closed sets $F_1$ and $F_2$ of $X$, there exist disjoint $ge^{\ast }\theta$-open sets $U_1$ and $U_2$ such that $F_1 \subseteq U_1$, $F_2 \subseteq U_2$.
\end{theorem}

\begin{proof}
$(1) \Rightarrow (2):$ The proof is clear from Theorem \ref{9} since every $e^{\ast }$-$\theta$-closed set is $ge^{\ast }\theta$-closed.\\

$(2) \Rightarrow (3):$ Let $F\in C(X), \ G \in O(X)$ and $F \subseteq G.$ \\
 $\left.
	\begin{array}{rr}
(F\in C(X))(G \in O(X))(F\subseteq G)\Rightarrow (\setminus F,G\in O(X))((\setminus F)\cup G=X)\\
 \text{Hypothesis}
 \end{array} \right\} \Rightarrow$ \\ 
 $\begin{array}{l}
 \Rightarrow (\exists A,B\in ge^{*} \theta C(X))(A\subseteq \setminus F)(B\subseteq G)(A\cup B=X)
 \end{array}$ \\
 $\left.
	\begin{array}{rr}
	\Rightarrow (\setminus A, \setminus B\in ge^{*} \theta O(X))(F\subseteq \setminus A)(\setminus G \subseteq \setminus B)((\setminus A)\cap (\setminus B)=\emptyset)
\\
 (U:=e^*\text{-}int_\theta(\setminus A))(V:=e^*\text{-}int_\theta(\setminus B))
 \end{array} \right\} \Rightarrow$ \\ 
 $\begin{array}{l}
    \Rightarrow (U,V\in e^{*} \theta O(X))(F\subseteq U)(\setminus G\subseteq V)(U\cap V=\emptyset)
 \end{array}$ \\
 $\begin{array}{l}
    \Rightarrow (U\in ge^{*} \theta O(X))(F\subseteq U\subseteq e^{*}\text{-}cl_\theta(U) \subseteq \setminus V \subseteq G).
 \end{array}$ \\
	
$(3) \Rightarrow (4):$ Let $F_1,F_2 \in C(X)$ and $F_1 \cap F_2 = \emptyset$.\\
	$
	\left.
	\begin{array}{r}
	(F_1,F_2 \in C(X)) (F_1 \cap F_2 = \emptyset )\Rightarrow  ( \setminus F_1,\setminus F_2 \in O(X)) (F_1 \subseteq  \setminus F_2)\\ 
	 \text{Hypothesis}
	\end{array}
	\right\} \Rightarrow
	$
	\\
	$
	\left.
	\begin{array}{r}
    \Rightarrow (\exists U \in ge^* \theta O(X))(F_1 \subseteq U \subseteq e^{\ast}\mbox{-}cl_{\theta}(U)\subseteq \setminus F_2) \\ 
	V:=\setminus e^{\ast}\mbox{-}cl_{\theta}(U) 
	\end{array}
	\right\} \Rightarrow 
	$
	\\
	$\begin{array}{r}
	\Rightarrow (U,V \in ge^* \theta O(X))(F_1 \subseteq U)(F_2 \subseteq V)(U \cap V=\emptyset).
	\end{array}$
	\\	

$(4) \Rightarrow (1):$ Let $F_1,F_2 \in C(X)$ and $F_1 \cap F_2 = \emptyset$.\\
	$
	\left.
	\begin{array}{r}
	(F_1,F_2 \in C(X)) 	(F_1 \cap F_2 = \emptyset) \\ 
    \text{Hypothesis}
	\end{array}
	\right\} \Rightarrow
	$
	\\
	$
	\left. \begin{array}{rr}
    \Rightarrow (\exists U_1,U_2 \in ge^* \theta O(X))(F_1 \subseteq U_1)(F_2 \subseteq U_2)(U_1 \cap U_2=\emptyset)\\
    \text{Lemma } \ref{6} 
	\end{array} \right\} \Rightarrow
	$
	\\
	$
	\left.
	\begin{array}{r}
    \Rightarrow (F_1 \subseteq e^{\ast}\mbox{-}int_{\theta}(U_1))(F_2 \subseteq e^{\ast}\mbox{-}int_{\theta}(U_2)) (U_1 \cap U_2=\emptyset) \\ 
    (V_1:=e^{\ast}\mbox{-}int_{\theta}(U_1))(V_2:=e^{\ast}\mbox{-}int_{\theta}(U_2))
	\end{array}
	\right\} \Rightarrow
	$
	\\
	$\begin{array}{r}
	\Rightarrow (V_1,V_2 \in e^* \theta O(X))(F_1 \subseteq V_1)(F_2 \subseteq V_2)(V_1 \cap V_2=\emptyset).
	\end{array}$
\end{proof}

\begin{definition} 
A topological space is said to be $e^*\theta$-$R_0$ if every open set of the space contains the $e^*$-$\theta$-closure of each of  its singletons.
\end{definition}

\begin{theorem}
    If $X$ is $e^*\theta$-normal and $e^*\theta$-$R_0$, then $X$ is $e^*\theta$-regular.
\end{theorem}

\begin{proof}
Let $F \in C(X)$ and $x \notin F$.\\
$
\left.\begin{array}{rr} 
x \notin F \in C(X) \Rightarrow  \setminus F \in O(X,x)\\X \text{ is } e^*\theta\text{-}R_0  \end{array}	\right\}\Rightarrow \!\!\!\!\! \begin{array}{c} \\	\left. \begin{array}{r} K:= e^{\ast}\mbox{-}cl_{\theta}(\{x\}) \subseteq  \setminus F  \\ X \text{ is } e^*\theta\text{-normal}\end{array}	\right\} \Rightarrow \end{array}
$	\\
	$\begin{array}{r}
	\Rightarrow (\exists U \in e^* \theta O(X,x))(\exists V \in e^* \theta O(X))(K \subseteq U)(F \subseteq V)(U \cap V=\emptyset).
	\end{array}$
\end{proof}

\begin{definition} 
\!\! A topological space is said to be almost $e^*\theta$-irresolute if $f[e^{\ast}\mbox{-}cl_{\theta}(U)]= e^{\ast}\mbox{-}cl_{\theta}(f[U])$ for every $U \in e^* \theta O(X)$.
\end{definition}

\begin{theorem}
    If $f:X \rightarrow Y$ is a pre $e^*\theta$-open continuous almost $e^*\theta$-irresolute surjection from an $e^*\theta$-normal space $X$ onto a space $Y$, then $Y$ is $e^*\theta$-normal.
\end{theorem}

\begin{proof}
Let $A\in C(Y)$, $B \in O(Y)$ and $A \subseteq B$.\\
$
	\left.
	\begin{array}{rr}
	(A\in C(Y))(B \in O(Y))(A \subseteq B) \\ 
	f \text{ is continuous} \end{array} \right\} \Rightarrow $
	\\
	$
	\left.
	\begin{array}{rr}
	\Rightarrow (f^{-1}[A] \in C(X))(f^{-1}[B]  \in O(X))(f^{-1}[A] \subseteq f^{-1}[B] ) \\
    X \text{ is } e^*\theta\text{-normal} 
	\end{array}
	\right\} \overset{\text{Theorem \ref{9}(3)}}{\Rightarrow}
	$
	\\
	$
	\left.
	\begin{array}{r}
    \Rightarrow   (\exists U \in e^*\theta O(X))(f^{-1}[A] \subseteq U \subseteq e^{\ast}\mbox{-}cl_{\theta}(U) \subseteq f^{-1}[B]) \\ 
    f \text{ is pre } e^*\theta\text{-open}
	\end{array}
	\right\}  \Rightarrow
	$
	\\
	$\left. \begin{array}{rr}
	\Rightarrow (f[U] \in e^* \theta O(Y))(f[f^{-1}[A]] \subseteq f[U] \subseteq f[e^{\ast}\mbox{-}cl_{\theta}(U)] \subseteq f[f^{-1}[B]])\\
	f \text{ is surjective}
	\end{array}\right\} \Rightarrow $  \\
	$
	\left.
	\begin{array}{r}
    \Rightarrow  (f[U] \in e^*\theta O(Y))(A \subseteq f[U] \subseteq f[e^{\ast}\mbox{-}cl_{\theta}(U)] \subseteq B) \\ 
    f \text{ is almost } e^*\theta\text{-irresolute}
	\end{array}
	\right\}  \Rightarrow
	$
	\\
	$\begin{array}{r}
	\Rightarrow (f[U] \in e^* \theta O(Y))(A \subseteq f[U] \subseteq e^{\ast}\mbox{-}cl_{\theta}(f[U]) \subseteq B)
	\end{array}$\\
Then, $Y$ is $e^*\theta$-normal from Theorem \ref{9}(3).
\end{proof}

\begin{theorem}
    If $f:X \rightarrow Y$ is a $ge^*\theta$-closed (pre $ge^*\theta$-closed) continuous function from a normal ($e^*\theta$-normal) space $X$ onto a space $Y$, then $Y$ is $e^*\theta$-normal.
\end{theorem}

\begin{proof}
Let $A,B \in C(Y)$ and $A \cap B = \emptyset$.\\ 
	$
	\left.
	\begin{array}{rr}
	(A,B \in C(Y))(A \cap B = \emptyset) \\
	f \text{ is continuous} \end{array} \right\} \Rightarrow 
	$\\
	$
	\left.
	\begin{array}{rr}
	\Rightarrow (f^{-1}[A],f^{-1}[B] \in C(X))(f^{-1}[A] \cap f^{-1}[B] = f^{-1}[A \cap B] = \emptyset)  \\
	X \text{ is normal} \end{array} \right\} \Rightarrow
	$
	\\
	$
	\left.
	\begin{array}{rr}
    \Rightarrow   (\exists U,V \in O(X))(f^{-1}[A] \subseteq U)(f^{-1}[B] \subseteq V)(U \cap V=\emptyset) \\ 
    f \text{ is } ge^*\theta\text{-closed}
	\end{array}
	\right\}  \overset{\text{Lemma } \ref{7}}{\Rightarrow}
	$
	\\
	$\begin{array}{l}
	\Rightarrow (\exists G,H \in ge^* \theta O(Y))(A \subseteq G)(B \subseteq H)(f^{-1}[G] \subseteq U)(f^{-1}[H] \subseteq V)(G \cap H=\emptyset)
	\end{array}$\\	
Then, $Y$ is $e^*\theta$-normal from Theorem \ref{10}(4).\\
The other case can be proved similarly.
\end{proof}

\begin{definition}
	A topological space $X$ is said to be:\\
	$a)$ $e^{\ast }\theta $-$T_{0}$ \cite{ayhan-ozkoc2} if for
	any distinct pair of points $x$ and $y$ in $X,$ there is an $e^{\ast}$-$
	\theta $-open set $U$ in $X$ containing $x$ but not $y$ or an $e^{\ast }$-$\theta $-open set $V$ in $X$ containing $y$ but not $x$.
	\\
	$b)$ $e^{\ast }\theta $-$T_{1}$ \cite{ayhan-ozkoc2} if for any distinct pair of points $x$ and $y$ in $X,$ there is an $e^{\ast
	}$-$\theta $-open set $U$ in $X$ containing $x$ but not $y$ and an $e^{\ast }$-$\theta $-open set $V$ in $X$ containing $y$ but not $x$.
	\\
	$c)$ $e^{\ast }\theta $-$T_{2}$ \cite{ayhan-ozkoc2} $($resp. $e^{\ast }$-$T_{2}$ \cite{ekici4}$)$ if for every pair of distinct points $x$ and $y$, there exist two $e^{\ast }$-$\theta $-open (resp. $e^{\ast }$-open) sets $U$ and $V$ such that $x\in U,$ $\ y\in V$ and $U\cap V=$ $\emptyset.$
\end{definition}

\begin{theorem} 
	\cite{ayhan-ozkoc2} For a topological space $X,$ the following properties are equivalent:\\
	(1) $X$ is $e^{\ast }\theta $-$T_{0};$\\
	(2) $X$ is $e^{\ast }\theta $-$T_{1};$\\
	(3) $X$ is $e^{\ast }\theta $-$T_{2};$\\
	(4) $X$ is $e^{\ast }$-$T_{2};$\\
	(5) For every pair of distinct points $x,y\in X,$ there exist $U\in e^{\ast}O(X,x)$ and $V\in e^{\ast}O(X,y)$ such that $e^{\ast }$-$cl(U)\cap e^{\ast }$-$cl(V)=\emptyset;$\\
	(6) For every pair of distinct points $x,y\in X,$ there exist $U\in e^{\ast}R(X,x)$ and $V\in e^{\ast}R(X,y)$ such that $U\cap V=\emptyset;$\\
	(7) For every pair of distinct points $x,y\in X,$ there exist $U\in e^{\ast}\theta O(X,x)$ and $V\in e^{\ast }\theta O(X,y)$  such that $e^{\ast }$-$cl_{\theta }(U)\cap e^{\ast }$-$cl_{\theta }(V)=\emptyset.$
\end{theorem}


\begin{definition} \label{b}
	A function $f:X \rightarrow Y$ is said to be: \\
	$a)$ $e^{\ast}\theta$-continuous \cite{ayhan-ozkoc2} if $f^{-1}[V]$ is $e^{\ast}$-$ \theta$-closed in $X$ for every closed set $V$ in $Y$, equivalently if $f^{-1}[V]$ is $e^{\ast}$-$ \theta$-open in $X$ for every open set $V$ in $Y$.\\
	$b)$ strongly $e^{\ast}$-irresolute \cite{ozkoc-atasever} if for each point $x \in X$ and each $V \in e^*O(Y, f(x))$, there exists $U \in e^*O(X, x)$  such that $f[e^*\text{-}cl(U)] \subseteq V )$.
\end{definition}	

\begin{theorem} 
    A function $f:X \rightarrow Y$ is strongly $e^*$-irresolute if and only if $f^{-1}[V]$ is $e^{\ast}$-$ \theta$-open in $X$ for every $e^{\ast}$-$ \theta$-open set $V$ in $Y.$ 
\end{theorem}	

\begin{proof}
It is obvious from Theorem 5.1(e) \cite{ozkoc-atasever}.
\end{proof}

\begin{theorem} \label{x}
    If $f:X \rightarrow Y$ is a strongly $e^*$-irresolute closed injection from a space $X$ to an $e^*\theta$-normal space $Y$, then $X$ is $e^*\theta$-normal.
\end{theorem}

\begin{proof}
Let $A,B \in C(X)$ and $A \cap B = \emptyset $.\\
$
	\left.
	\begin{array}{rr}
	(A,B \in C(X))( A \cap B = \emptyset )\\
	f \text{ is closed injection} 
	\end{array}
	\right\} \Rightarrow 
$ \\
$
	\left.
	\begin{array}{rr}
	\Rightarrow (f[A],f[B]  \in C(Y))(  f[A] \cap f[B] = f[A \cap B] = \emptyset )\\
    Y \text{ is } e^*\theta\text{-normal} 
	\end{array}
	\right\} \Rightarrow 
$	\\
	$
	\left.
	\begin{array}{r}
    \Rightarrow (\exists U,V \in e^*\theta O(Y))(f[A] \subseteq U)(f[B] \subseteq V)(U \cap V = \emptyset)  \\ 
    f \text{ is strongly } e^*\text{-irresolute}
	\end{array}
	\right\}  \Rightarrow
	$
	\\
	$\begin{array}{r}
	\Rightarrow (f^{-1}[U],f^{-1}[V] \in e^* \theta O(X))(A \subseteq f^{-1}[U])(B \subseteq f^{-1}[V])(f^{-1}[U] \cap f^{-1}[V] = \emptyset).
	\end{array}$
\end{proof}

\begin{theorem}
    If $f:X \rightarrow Y$ is an $e^*\theta$-continuous closed injection from a space $X$ to a normal space $Y$, then $X$ is $e^*\theta$-normal.
\end{theorem}

\begin{proof}
It is similar to the proof of Theorem \ref{x}.
\end{proof}

\begin{definition} 
A topological space is said to be $(e^{\ast},\theta)^{\ast}$-normal if for any pair of disjoint $e^{\ast}$-$\theta$-closed sets $F_1$ and $F_2$ of $X$, there exist disjoint $e^*$-$\theta$-open sets $U_1$ and $U_2$ such that $F_1 \subseteq U_1$ and $F_2 \subseteq U_2$.
\end{definition}

\begin{definition} 
A subset $A$ of a space $X$ is said to be $(e^*,\theta)$-closed if $e^*$-$cl_{\theta}(A) \subseteq U$ whenever $A \subseteq U$ and $U$ is $e^*$-$\theta$-open in $X$. A subset $A$ of a space $X$ is said to be $(e^*,\theta)$-open if $X \setminus A$ is $(e^*,\theta)$-closed. The family of all $(e^*,\theta)$-closed (resp. $(e^*,\theta)$-open) subsets of $X$ is denoted by $(e^{\ast},\theta)C(X)$ $(\mbox{resp. }(e^{\ast},\theta)O(X)).$
\end{definition}

\begin{lemma} 
A subset $A$ of a space $X$ is $(e^*,\theta)$-open if and only if $F \subseteq e^*$-$int_{\theta}(A)$ whenever $F \subseteq A$ and $F$ is $e^*$-$\theta$-closed in $X$.
\end{lemma}

\begin{proof}
It is similar to the proof of Lemma \ref{6}.
\end{proof}

\begin{theorem} \label{00}
For a topological space $X$, the following statements are equivalent:\\
(1) $X$ is $(e^*,\theta)^{\ast}$-normal;\\
(2) For every pair of $e^*$-$\theta$-open sets $U$ and $V$ whose $U \cup V=X$, there exist $e^{\ast }$-$\theta$-closed sets $A$ and $B$ such that $A\subseteq U$, $B\subseteq V$ and $A \cup B=X$;\\
(3) For each $e^*$-$\theta$-closed set $F$ and every $e^*$-$\theta$-open set $G$ containing $F$, there exists an $e^*$-$\theta$-open set $U$ such that $F \subseteq U \subseteq e^{\ast}$-$cl_{\theta}(U)\subseteq G$;\\
(4) For every pair of $e^*$-$\theta$-open sets $U$ and $V$ whose $U \cup V=X$, there exist $(e^{\ast },\theta)$-closed sets $A$ and $B$ such that $A\subseteq U$, $B\subseteq V$ and $A \cup B=X$;\\
(5) For each $e^*$-$\theta$-closed set $F$ and every $e^*$-$\theta$-open set $G$ containing $F$, there exists an $(e^*,\theta)$-open set $U$ such that $F \subseteq U \subseteq e^{\ast}$-$cl_{\theta}(U)\subseteq G$;\\
(6) For each pair of disjoint $e^*$-$\theta$-closed sets $F_1$ and $F_2$ of $X$, there exist disjoint $(e^{\ast },\theta)$-open sets $U_1$ and $U_2$ such that $F_1 \subseteq U_1$, $F_2 \subseteq U_2$.
\end{theorem}

\begin{proof}
It is similar to the proof of Theorems \ref{9} and \ref{10}.
\end{proof}

\begin{definition} 
A function $f:X \rightarrow Y$ is said to be $(e^{\ast},\theta)^*$-closed, if the image of each $e^{\ast}$-$\theta$-closed set $F$ of $X$ is $(e^{\ast},\theta)$-closed in $Y$.
\end{definition}

\begin{lemma} \label{0}
A function $f:X \rightarrow Y$ is $(e^*,\theta)^*$-closed if and only if for each subset $B$ of $Y$ and each $e^*$-$\theta$-open set $U$ containing $f^{-1}[B]$, there exists an $(e^*,\theta)$-open set $V$ of $Y$ containing $B$ such that $f^{-1}[V] \subseteq U$.
\end{lemma}

\begin{proof}
It is similar to the proof of Lemma \ref{4}.
\end{proof}

\begin{theorem}
    If $f:X \rightarrow Y$ is $(e^*,\theta)^*$-closed and strongly $e^*$-irresolute surjection from an $(e^*,\theta)^*$-normal space $X$ onto a space $Y$, then $Y$ is $(e^*,\theta)^*$-normal.
\end{theorem}

\begin{proof}
Let $K,L \in e^*\theta C(Y)$ and $K \cap L = \emptyset $.\\
$\left.
	\begin{array}{r}
    (K,L \in e^*\theta C(Y))(K \cap L = \emptyset)\\ 
    f \text{ is strongly $e^*$-irresolute}
	\end{array}
	\right\}  \Rightarrow
$\\
$\left.
	\begin{array}{r}
    \Rightarrow (f^{-1}[K],f^{-1}[L]  \in e^*\theta C(X))(f^{-1}[K] \cap f^{-1}[L] = f^{-1}[K \cap L] = \emptyset  )\\ 
    X \text{ is } (e^*,\theta)^*\text{-normal}
	\end{array}
	\right\}  \Rightarrow
$\\
	$
	\left.
	\begin{array}{r}
    \Rightarrow (\exists U,V \in e^*\theta O(X))(f^{-1}[K] \subseteq U)(f^{-1}[L] \subseteq V)(U \cap V = \emptyset)  \\ 
    f \text{ is } (e^*,\theta)^*\text{-closed}
	\end{array}
	\right\}  \overset{\text{Lemma } \ref{0}}{\Rightarrow}
	$
	\\
	$
	\begin{array}{r}
	\Rightarrow (\exists G,H \in (e^*, \theta)O(Y))(K \subseteq G)(L \subseteq H)(f^{-1}[G] \subseteq U)(f^{-1}[H] \subseteq V)(G \cap H = \emptyset)
	\end{array}
	$\\
Then, $Y$ is $(e^*,\theta)^*$-normal from Theorem \ref{00}(6).
\end{proof}

\begin{theorem}
    If $f:X \rightarrow Y$ is pre $e^*\theta$-closed and strongly $e^*$-irresolute injection and $Y$ is $(e^*,\theta)^*$-normal, then $X$ is $(e^*,\theta)^*$-normal.
\end{theorem}

\begin{proof}
Let $K,L \in e^*\theta C(X)$ and $K \cap L = \emptyset $.\\
$\left.
	\begin{array}{r}
    (K,L \in e^*\theta C(X))(K \cap L = \emptyset)\\ 
    f \text{ is pre $e^*\theta$-closed injection}
	\end{array}
	\right\}  \Rightarrow
$\\
$\left.
	\begin{array}{r}
    \Rightarrow (f[K],f[L]  \in e^*\theta C(Y))(f[K] \cap f[L] = f[K \cap L] = \emptyset  )\\ 
    Y \text{ is } (e^*,\theta)^*\text{-normal}
	\end{array}
	\right\}  \Rightarrow
$	\\
	$
	\left.
	\begin{array}{r}
    \Rightarrow (\exists U,V \in e^*\theta O(Y))(f[K] \subseteq U)(f[L] \subseteq V)(U \cap V = \emptyset)  \\ 
    f \text{ is strongly } e^*\text{-irresolute}
	\end{array}
	\right\} \Rightarrow
	$
	\\
	$
	\begin{array}{r}
	\Rightarrow ( f^{-1}[U],f^{-1}[V] \in e^*\theta O(X))(A \subseteq f^{-1}[U])(B \subseteq f^{-1}[V])(f^{-1}[U] \cap f^{-1}[V]= \emptyset).
	\end{array}
	$
\end{proof}


\section{Conclusion}
 Many forms of regular and normal spaces have been studied by many authors in recent years. This paper is concerned with the notion of $e^{\ast}\theta$-regular and $e^{\ast}\theta$-normal spaces which are defined by utilizing the concept of $e^{\ast}$-$\theta$-open set. 
We have seen that these concept is weaker and stronger than many generalized regular and normal space forms in the literature, as will be seen in Remark \ref{eb} and Remark \ref{be}. 
We believe that this work will shed light on further studies related to continuity and convergence etc. known from functional analysis.

\section{Acknowledgements} 
This study has been supported by Turkish–German University Scientific Research Projects Commission under the grant no: 2021BR01. In addition,
I would like to thank to Professor Dr. Murad ÖZKOÇ for his valuable suggestions and comments that improved the study.

\end{document}